\documentclass[12pt,reqno]{amsart}
\usepackage{fullpage}
\usepackage{times}
\usepackage{colonequals}
\usepackage{amsmath,amssymb,amsthm,url}
\usepackage[utf8]{inputenc}
\usepackage[english]{babel}
\usepackage{comment}
\usepackage{bbm}
\usepackage{enumerate}
\usepackage{bm}
\usepackage{graphicx}
\usepackage{mathrsfs}
\usepackage[colorlinks=true, pdfstartview=FitH, linkcolor=blue, citecolor=blue, urlcolor=blue]{hyperref}

\newtheorem{thm}{Theorem}[section]
\newtheorem{lem}[thm]{Lemma}

\newtheorem{cor}[thm]{Corollary}

\newtheorem{claim}[thm]{Claim}
\newenvironment{poc}{\begin{proof}[Proof of claim]}{\end{proof}}

\theoremstyle{definition}

\newtheorem{ex}[thm]{Example}

\newcommand{\F}{\mathbb{F}}

\title{Extensions of the Carlitz--McConnel and Blokhuis--Sziklai theorems for unions of cyclotomic classes}
\author{Maosheng Xiong}
\address{Department of Mathematics, Hong Kong University of Science and Technology, Clear Water Bay, Hong Kong}
\email{mamsxiong@ust.hk}
\author{Chi Hoi Yip}
\address{Department of Mathematics, Hong Kong University of Science and Technology, Clear Water Bay, Hong Kong}
\email{machyip@ust.hk}

\begin{document}
\begin{abstract}
Let $p$ be a prime, let $q=p^n$, and let $D\subseteq \mathbb{F}_q^\ast$. A celebrated result of Carlitz and McConnel states that if $D$ is a proper subgroup of $\mathbb{F}_q^\ast$, and $f:\mathbb{F}_q\to\mathbb{F}_q$ is a function such that $(f(x)-f(y))/(x-y)\in D$ for all $x\neq y$, then $f$ must be of the form $f(x)=ax^{p^j}+b$. In this paper, we extend their result to the setting where $D$ is a union of cosets of a fixed subgroup of $\mathbb{F}_q^\ast$, under a mild assumption. In a similar spirit, we also investigate maximum cliques in related Cayley graphs over finite fields, strengthening several results of Blokhuis, Sziklai, and Asgarli and Yip.
\end{abstract}
\maketitle

\section{Introduction}
Throughout this paper, let \(p\) be a prime and \(n\) a positive integer. Let \(q=p^n\), and let \(\F_q\) denote the finite field with \(q\) elements. We write \(\F_q^\ast=\F_q\setminus\{0\}\).

A celebrated result of Carlitz and McConnel \cite{C60, M62} states the following.

\begin{thm}[\cite{C60,M62}]\label{thm:CM}
Let \(p\) be a prime and \(n\) a positive integer, and let \(q=p^n\). If \(D\) is a proper subgroup of \(\F_q^\ast\), and \(f:\F_q\to\F_q\) is a function such that
\begin{equation}\label{eq:f}
\frac{f(x)-f(y)}{x-y}\in D
\end{equation}
whenever \(x,y\in\F_q\) with \(x\neq y\), then there exist \(a,b\in\F_q\) and an integer \(0\le j\le n-1\) such that
\[
f(x)=a x^{p^j}+b \qquad \text{for all } x\in\F_q.
\]
\end{thm}

Theorem~\ref{thm:CM} was first proved by Carlitz \cite{C60} in the case where \(q\) is odd and \(D\) is the subgroup of index \(2\), and was later extended by McConnel \cite{M62} to all proper subgroups. The Carlitz--McConnel theorem has various connections with finite geometry, graph theory, and group theory; we refer to the survey of Jones \cite[Section~9]{J20}. In particular, it has many applications in finite geometry; see, for example, \cite{BSW92,KS10}. We also refer to \cite{BL73,G81,L90} for variants of the Carlitz--McConnel theorem proved using group-theoretic tools.

Theorem~\ref{thm:CM} states that the condition~\eqref{eq:f} forces $f$ to be an automorphism of the field, under some normalization. Here we intentionally state Theorem~\ref{thm:CM} in a form different from the original formulations of Carlitz \cite{C60} and McConnel \cite{M62}, who expressed the result in terms of multiplicative characters. The main reason is that our formulation connects more naturally with the theory of directions in finite geometry, which we briefly review next.

Let \(AG(2,q)\) denote the \emph{affine Galois plane} over \(\F_q\). Let \(U\subseteq AG(2,q)\) be a set of points, and write
$U=\{(x_i,y_i):1\le i\le |U|\}.$ The set of \emph{directions determined by} \(U\) is
\[ \mathcal{D}_U=\left\{ \frac{y_j-y_i}{x_j-x_i} \colon 1\leq i <j \leq |U| \right \} \subseteq \F_q \cup \{\infty\},\]
where \(\infty\) denotes the vertical direction. If \(f:\F_q\to\F_q\) is a function, we may consider its graph
$U(f)=\{(x,f(x)):x\in\F_q\},$
and define the set of \emph{directions determined by \(f\)} by \(\mathcal D_f:=\mathcal D_{U(f)}\). Since \(U(f)\) is a graph, we have \(\infty\notin \mathcal D_f\), and \(\mathcal D_f\) is precisely the set of slopes of secant lines through pairs of points on \(U(f)\). In this language, condition~\eqref{eq:f} is equivalent to \(\mathcal D_f\subseteq D\).

The following well-known result is due to Blokhuis, Ball, Brouwer, Storme, and Sz{\H{o}}nyi \cite{B03,BBBSS99}.

\begin{thm}[\cite{B03, BBBSS99}]\label{thm:directions}
Let $p$ be a prime and let $q=p^n$. Let $f:\F_q \to \F_q$ be a function such that $f(0)=0$. If $|\mathcal{D}_f|\leq \frac{q+1}{2}$, then $f$ is additive, that is, $f(x+y)=f(x)+f(y)$ for all $x,y\in \F_q$.
\end{thm}
In particular, when \(q=p\) is prime, Theorem~\ref{thm:directions} implies Theorem~\ref{thm:CM} immediately; this was first observed by Lov\'asz and Schrijver \cite{LS83}. In the general prime-power case, however, Theorem~\ref{thm:directions} does not directly imply Theorem~\ref{thm:CM}. Nevertheless, further connections between the theory of directions and the Carlitz--McConnel theorem have recently been observed by Muzychuk \cite{M20}, Aguglia--Csajb\'ok--Weiner \cite{ACW24}, the second author \cite{Y25}, and Csajb\'ok \cite{B25}. In particular, the second author \cite{Y25} proved the following strengthening of Theorem~\ref{thm:CM}.

\begin{thm}[\cite{Y25}]\label{thm:Y}
Let \(p\) be a prime, let \(q=p^n\), and let \(D\subseteq \F_q^\ast\). Suppose \(f:\F_q\to\F_q\) is a function such that \(\mathcal D_f\subseteq D\). If
\[
|DD^{-1}D^{-1}|\le \frac{q+1}{2},
\]
then there exist \(a,b\in\F_q\) and an integer \(0\le j\le n-1\) such that
\[
f(x)=a x^{p^j}+b \qquad \text{for all } x\in\F_q.
\]
\end{thm}

Note that when \(D\) is a proper subgroup of \(\F_q^\ast\), we have \(DD^{-1}D^{-1}=D\) and \(|D|\le (q-1)/2\), so Theorem~\ref{thm:Y} recovers Theorem~\ref{thm:CM}.

\subsection{An extension of the Carlitz-McConnel theorem}

In this paper, we aim to characterize functions \(f:\F_q\to\F_q\) with \(\mathcal D_f\subseteq D\), where \(D\) is a union of cosets of a fixed subgroup of \(\F_q^\ast\). This study is motivated by two main considerations.

Our first motivation comes from arithmetic combinatorics. It is remarked in \cite[Theorem~1.2]{Y25} that Theorem~\ref{thm:Y} applies in particular when \(|D|\) is small and \(D\) has small multiplicative doubling (that is, \(|DD|/|D|\) is small). Heuristically, this means that \(D\) should resemble a multiplicative subgroup of \(\F_q^\ast\), whose doubling constant is \(1\). More generally, by the Freiman theorem for abelian groups due to Green and Ruzsa \cite{GR07}, sets of small doubling admit strong structural control in terms of coset progressions. This makes it natural to study the case where \(D\) is a union of cosets of a fixed subgroup $H$ of $\F_q^\ast$.

Thus, it is natural to extend Theorem~\ref{thm:Y} to the setting in which \(|D|\le (q+1)/2\) and \(D\) is a union of cosets of a fixed subgroup of \(\F_q^\ast\). We establish this in Theorem~\ref{thm:main}, and deduce a convenient consequence in Corollary~\ref{cor:main}, under a mild additional assumption. In this more general setting, it is easy to construct examples for which \(DD^{-1}D^{-1}=\F_q^\ast\) (so Theorem~\ref{thm:Y} does not apply). For example, let \(H\) be a subgroup of \(\F_q^\ast\) of index \(7\), let \(g\) be a primitive root of \(\F_q^\ast\), and set
\[
D=H\cup gH\cup g^2H.
\]

Our second main motivation comes from algebraic graph theory. Although the original proofs of Theorem~\ref{thm:CM} by Carlitz and McConnel \cite{C60,M62} do not involve graph theory, interestingly, Carlitz's theorem \cite{C60} is perhaps more famously known as the characterization of the automorphism group of Paley graphs; see \cite[Section 9]{J20}. Recall that if \(q\equiv 1\pmod 4\), the \emph{Paley graph of order \(q\)}, denoted \(P_q\), is the graph with vertex set \(\F_q\), in which two vertices are adjacent if their difference is a nonzero square in \(\F_q\). Equivalently, if \(\chi\) is the quadratic character of \(\F_q\), then a permutation \(f:\F_q\to\F_q\) is an automorphism of \(P_q\) if and only if
\[
\chi(f(x)-f(y))=\chi(x-y)\qquad \text{for all }x\neq y.
\]
This is equivalent to condition~\eqref{eq:f} with
$D=\{x^2:x\in\F_q^\ast\}.$ Thus, Carlitz's theorem \cite{C60} characterizes the automorphism group of Paley graphs. More generally, Theorem~\ref{thm:CM} characterizes automorphism groups of cyclotomic schemes \cite{M20}. Indeed, if $H$ is a subgroup of $\F_q^*$, then an automorphism \(f:\F_q\to\F_q\) of the cyclotomic scheme corresponding to \(H\) is precisely a permutation such that whenever \(x\neq y\), \(f(x)-f(y)\) and \(x-y\) always lie in the same coset of \(H\). We remark that the corresponding question over generalized Paley graphs, namely characterizing the automorphism group of a generalized Paley graph remains widely open; some partial progress has been made in \cite[Theorem 1.2]{LP09}.

In the more general setting where \(D\) is a union of cosets of a fixed subgroup of \(\F_q^\ast\), classifying functions \(f:\F_q\to\F_q\) satisfying~\eqref{eq:f} is equivalent to characterizing automorphisms of certain fusion schemes of cyclotomic schemes (see \cite{FM13, FX12, X12} for many constructions of such schemes). For example, we believe the following question is of independent interest: Let $H$ be a subgroup of $\F_q^*$ of index $d\geq 4$ and let $g$ be a primitive root of $\F_q$. Classify permutations $f:\F_q\to \F_q$ such that for each $0\leq i\leq d-1$,
\[x-y\in g^i H \implies f(x)-f(y)\in g^{i-1}H \cup g^i H \cup g^{i+1}H.\] 
Equivalently, this asks for a characterization of the permutations $f$ satisfying~\eqref{eq:f} with \[D=g^{-1}H \cup H \cup gH.\] 
We address this question in our new results. 

The following is our main result, which essentially addresses the above two motivations, under some mild assumption. 

\begin{thm}\label{thm:main}
Let $n,d,r$ be integers with $n,d\geq 2$ and $1\leq r\leq d-1$. Let $q=p^n$, where $p$ is a prime such that $p^n \equiv 1 \pmod d$ and \[p\geq \frac{(2n-1)^2 rd^2}{(d-r)^2}.\] 
Let $H$ be the subgroup of $\F_q^*$ with index $d$ and let $D$ be a union of $r$ cosets of $H$. If $f:\F_q \to \F_q$ is a function such that $\mathcal{D}_f \subseteq D$ and $|\mathcal{D}_f|\leq \frac{q+1}{2}$, then there are $a,b \in \F_q$ and an integer $0 \leq j \leq n-1$, such that 
\[
f(x)=a x^{p^j}+b \qquad \text{for all } x\in\F_q.
\]
\end{thm}

To the best knowledge of the authors, none of the existing proofs of Theorem~\ref{thm:CM} extend to the setting of union of cosets in Theorem~\ref{thm:main}. Our proof of Theorem~\ref{thm:main} is based on a novel combination of tools from character sum estimates and finite geometry. We note that when $n=1$, the theorem follows from Theorem~\ref{thm:directions}. Thus, we focus on the case $n\geq 2$. 

We have the following immediate corollary of Theorem~\ref{thm:main}. 

\begin{cor}\label{cor:main}
Let $n,d$ be integers with $n,d\geq 2$. Let $q=p^n$, where $p$ is a prime such that $p^n \equiv 1 \pmod d$ and $p\geq 2(2n-1)^2 d.$ Let $H$ be the subgroup of $\F_q^*$ with index $d$ and let $D$ be a union of at most $d/2$ cosets of $H$. If $f:\F_q \to \F_q$ is a function such that $\mathcal{D}_f \subseteq D$, then there are $a,b \in \F_q$ and an integer $0 \leq j \leq n-1$, such that 
\[
f(x)=a x^{p^j}+b \qquad \text{for all } x\in\F_q.
\]
\end{cor}

We remark that the lower bound assumption on $p$ in Corollary~\ref{cor:main} (and thus Theorem~\ref{thm:main}) is essential in general. Indeed, if this assumption is removed, then the conclusion can fail even for very small fields, as the following example shows.

\begin{ex}
Consider the field $\F_{25}$. Let $u\in \F_{25}$ such that $u^2=2$; note that $u\notin \F_5$. Consider the subgroup $H=\F_5^*$ with index $d=6$. Define the following union of $3$ cosets of $H$:
\[
D:=uH\ \cup\ (1+u)H\ \cup\ (1-u)H.
\]
Consider the function
\[
f(x)=x+u\,x^5, \qquad \forall x\in \F_{25}.
\]
Then clearly $f$ is not of the desired form $ax^{5^j}+b$. However, we will show below that $\mathcal{D}_f \subseteq D$.

For distinct $x,y\in\mathbb F_{25}$, we have
\[
\frac{f(x)-f(y)}{x-y}
=
\frac{(x-y)+u(x^5-y^5)}{x-y}
=
1+u(x-y)^4.
\]
A direction computation shows that 
$(\mathbb F_{25}^*)^4=\{1,-1,2+2u,2-2u,3+2u,3-2u\}$ and thus
\[
D_f
=
\{\,2u,\ 3u,\ 1+u,\ 1-u,\ 2+2u,\ 2-2u\,\}.
\]
It follows that $D_f\subseteq D$ since 
\[
\{2u,3u\}\subset uH,\qquad
\{1+u,2+2u\}\subset (1+u)H,\qquad
\{1-u,2-2u\}\subset (1-u)H.
\]
\end{ex}

\subsection{An extension of the Blokhuis--Sziklai theorem and the Asgarli--Yip theorem}
Inspired by the proof of Theorem~\ref{thm:main}, we also make new progress on a relevant question lying in the intersection of arithmetic combinatorics and algebraic graph theory. 

We first review some necessary backgrounds. The following theorem, due to Blokhuis \cite{B84} and Sziklai \cite{Szi99}, is well-known.
\begin{thm}[\cite{B84, Szi99}]\label{thm:BS}
Let $q$ be a prime power. Let \(S\) be a subgroup of \(\F_{q^2}^\ast\) with index $d\geq 2$, where $d\mid (q+1)$. If $A\subseteq \F_{q^2}$ with $|A|=q$ and $0,1\in A$ such that \[A-A=\{a-b: a,b\in A\} \subseteq S \cup \{0\},\] then $A$ is the subfield $\F_q$. 
\end{thm}

Theorem~\ref{thm:BS} was first proved by Blokhuis \cite{B84} in the case where \(q\) is odd and \(S\) is the index-\(2\) subgroup, confirming a conjecture of van Lint and MacWilliams \cite{vLM78}. It was later extended by Sziklai \cite{Szi99} to all proper subgroups \(S\) of index \(d\mid (q+1)\). We remark that the divisibility condition \(d\mid (q+1)\) is necessary; indeed, if \(d\nmid (q+1)\), then no such set \(A\) exists (see \cite[Theorem~1.4]{Y26}). Theorem~\ref{thm:BS} may be viewed as an instance of the sum--product phenomenon, which broadly predicts that a large subset of a finite field cannot simultaneously exhibit strong additive and multiplicative structure unless it is ``close to'' a subfield; see \cite{T09} for a comprehensive discussion.

In the language of algebraic graph theory, Theorem~\ref{thm:BS} is also known as a characterization of maximum cliques in certain generalized Paley graphs of square order, and equivalently as an Erd\H{o}s--Ko--Rado (EKR) theorem for such graphs; see \cite{AY22, Y24, Y26} for further discussion.

We note that although the proof techniques for Theorem~\ref{thm:CM} and Theorem~\ref{thm:BS} are quite different, the two results are similar in the sense that they both exhibit rigidity phenomena over finite fields. Following a similar line of investigation, it is interesting to study the extension of Theorem~\ref{thm:BS} to general subsets $S$ that ``behave like" subgroups \cite{AY22, NM, Y24, Y26}. Indeed, recently, the second author \cite{Y26} extended Theorem~\ref{thm:BS} by proving an analogue of Theorem~\ref{thm:Y} in the setting of Theorem~\ref{thm:BS}.

Inspired by the proof of Theorem~\ref{thm:main}, we establish the following theorem, which can be viewed as an analogue of Theorem~\ref{thm:main} in the setting of Theorem~\ref{thm:BS}. More precisely, we extend Theorem~\ref{thm:BS} to the setting where $S$ is given by a union of cosets of a subgroup $H$ of $\F_{q^2}^*$ with index $d \mid (q+1)$. We note that none of the existing proofs of Theorem~\ref{thm:BS} appear to extend to this general setting.

\begin{thm}\label{thm:main2}
Let $n,d,r$ be integers with $n,d\geq 2$ and $1\leq r\leq d/2$. Let $q=p^n$, where $p$ is a prime such that $q \equiv -1 \pmod d$ and 
\begin{equation}\label{eq:lbp}
p\geq \frac{(4n-1)^2 rd^2}{(d-r)^2}.    
\end{equation}
Let $H$ be the subgroup of $\F_{q^2}^*$ with index $d$. Let $S$ be a union of $r$ cosets of $H$, such that $\F_q^* \subseteq S$. If $A \subseteq \F_{q^2}$ such that $0,1\in A$, $|A|=q$, and $A-A \subseteq S \cup \{0\}$, then $A=\F_q$.
\end{thm}

While the lower bound on $p$ in inequality~\eqref{eq:lbp} might not be optimal, we remark that Theorem~\ref{thm:main2} fails if there is no lower bound assumption on $p$; we refer to \cite[Example 2.21]{AY22} and \cite{GY24} for related discussions. Again we focus on the case $n\geq 2$: when $n=1$, it is known that Theorem~\ref{thm:main2} follows from Theorem~\ref{thm:directions}; see for example \cite[Corollary 2.14]{AY22}.

Theorem~\ref{thm:main2} refines the main result of Asgarli and the second author
\cite[Theorem~1.3]{AY22}. Its main advantage is that the required lower bound on \(p\)
is uniform, \emph{i.e.}, independent of the choice of \(S\). By contrast, in
\cite[Theorem~1.3]{AY22} the lower bound on \(p\) depends on \(S\); moreover, their
theorem does not always apply to such sets \(S\), since it imposes an (undesired)
``\(\varepsilon\)-lower-bounded" assumption on \(S\).

The following corollary follows from Theorem~\ref{thm:main2} immediately.

\begin{cor}\label{cor:main2}
Let $n,d$ be integers with $n,d\geq 2$. Let $q=p^n$, where $p$ is a prime such that $q \equiv -1 \pmod d$ and \(p\geq 2(4n-1)^2 d.\) 
Let $H$ be the subgroup of $\F_{q^2}^*$ with index $d$. Let $S$ be a union of at most $d/2$ cosets of $H$, such that $\F_q^* \subseteq S$. If $A \subseteq \F_{q^2}$ such that $0,1\in A$, $|A|=q$, and $A-A \subseteq S \cup \{0\}$, then $A=\F_q$.    
\end{cor}

In particular, this improves and extends \cite[Theorem 2.20]{AY22} by Asgarli and the second author, where the same result was proved for even numbers $d\geq 4$ and primes $p>4.1n^2d^4/\pi^2(d-1)^2$ with 
$S=\bigcup_{j=0}^{d/2-1} g^j H,$
where $g$ is a primitive root of $\F_{q^2}$. Note that the hypothesis in \cite{AY22} imposes a lower bound on $p$ of order $n^2d^2$, whereas Corollary~\ref{cor:main2} requires only a lower bound of order $n^2d$.

\subsection*{Outline of the paper}
In Section~\ref{sec:prelim}, we prove several preliminary lemmas. In Section~\ref{sec:main1}, we present the proof of Theorem~\ref{thm:main}. Finally, in Section~\ref{sec:main2}, we use a similar approach to prove Theorem~\ref{thm:main2}.

\section{Preliminaries}\label{sec:prelim}
The following lemma provides an estimate for certain character sums over subfields. 

\begin{lem}\label{lem:charactersum}
Let $p$ be a prime and $n$ a positive integer. Let $\xi_1,\xi_2\in \F_{p^n}$ be such that they are not Galois conjugates over $\F_p$, that is, $\xi_1\neq \xi_2^{p^r}$ for any $r\in \{0,1,\ldots, n-1\}$. Let $\chi_1$ and $\chi_2$ be multiplicative characters of $\F_{p^n}$ such that there exists $j\in \{1,2\}$ such that $\chi_j$ is not identically~$1$ on the set $\F_p[\xi_j] \setminus \{0\}$. Then
$$\bigg |\sum_{\lambda\in\mathbb{F}_p}\chi_1\big(\lambda-\xi_1\big) \chi_2\big(\lambda-\xi_2\big)\bigg|\le(2n-1)\sqrt p.$$
\end{lem}
\begin{proof}
This follows immediately from \cite[Corollary 3.5]{MY25} (a corrected version of \cite[Corollary 2.4]{W97}) by setting $f_1(T)=T-\xi_1$ and $f_2(T)=T-\xi_2$. 
\end{proof}

Next, we use Lemma~\ref{lem:charactersum} to deduce the following corollary, which is crucial in the proof of Theorem~\ref{thm:main}.
\begin{cor}\label{cor:charsum}
Let $p$ be a prime and $n$ a positive integer. Let $\chi$ be a nontrivial multiplicative character of $\F_{p^n}$. Let $a,b\in \F_{p^n}^*$ such that 
\[\chi(a/b)\neq 1, \qquad \text{and} \qquad (a/b)^{(p^n-1)/(p-1)}\neq 1.
\]
Let $u,v\in \F_{p^n}^*$ such that $u/v\notin \F_p$. Then
$$\bigg |\sum_{\lambda \in\mathbb{F}_p}\chi \bigg(\frac{au-b\lambda v}{u-\lambda v}\bigg)\bigg|\le(2n-1)\sqrt p.$$ 
\end{cor}
\begin{proof}
Let $d$ be the order of $\chi$. Let $\chi_1=\chi$ and $\chi_2=\chi^{d-1}$. Let $\xi_1=au/bv$ and $\xi_2=u/v$. 

We claim that $\xi_1$ and $\xi_2$ are not Galois conjugates over $\F_p$. Suppose otherwise that $\xi_1=\xi_2^{p^r}$ for some $r\in \{0,1,\ldots, n-1\}$. Then we have \[a/b=\xi_1/\xi_2=\xi_2^{p^r-1}=(\xi_2^{p-1})^{\frac{p^r-1}{p-1}}\] and it follows that \[(a/b)^{(p^n-1)/(p-1)}=(\xi_2^{p^n-1})^{\frac{p^r-1}{p-1}}=1,\] violating the given assumption. 

We also claim that there exists $j\in \{1,2\}$ such that $\chi_j$ is not identically~$1$ on the set $\F_p[\xi_j] \setminus \{0\}$. Suppose otherwise; then in particular we have $\chi_1(\xi_1)=\chi_2(\xi_2)=1$. It follows that  $\chi(\xi_1)=\chi(\xi_2)=1$ and $\chi(a/b)=\chi(\xi_1/\xi_2)=1$, violating the given assumption.

Thus, we can apply Lemma~\ref{lem:charactersum} to conclude that
\begin{align*}
\bigg |\sum_{\lambda \in\mathbb{F}_p}\chi \bigg(\frac{au-b\lambda v}{u-\lambda v}\bigg)\bigg|
&=\bigg |\sum_{\lambda \in\mathbb{F}_p}\chi_1(au-b\lambda v)\chi_2(u-\lambda v)\bigg|\\
&=\bigg |\sum_{\lambda \in\mathbb{F}_p}\chi_1(\lambda-\xi_1)\chi_2(\lambda-\xi_2)\bigg|\leq (2n-1)\sqrt{p}.\qedhere
\end{align*}
\end{proof}

Similarly, we have the following corollary. 

\begin{cor}\label{cor:charsum2}
Let $p$ be a prime and $n$ a positive integer. Let $\chi$ be a nontrivial multiplicative character of $\F_{p^{2n}}$. Let $a\in \F_{p^n}^*$ and $b\in \F_{p^{2n}}$ such that $b$ does not lie in any proper subfield of $\F_{p^{2n}}$. Then
$$\bigg |\sum_{\lambda \in\mathbb{F}_p}\chi \bigg(\frac{b-\lambda}{a-\lambda}\bigg)\bigg|\le(4n-1)\sqrt p.$$ 
\end{cor}
\begin{proof}
Note that $\F_p[b]=\F_{p^{2n}}$ and $b\neq a^{p^r}$ for any integer $r$, since $a\in \F_{p^n}^*$ and $b$ does not lie in any proper subfield of $\F_{p^{2n}}$. Thus, similar to the proof of Corollary~\ref{cor:charsum}, the required estimate follows from Lemma~\ref{lem:charactersum}.
\end{proof}

The following lemma will be useful in locating $b$ satisfying the condition in Corollary~\ref{cor:charsum2}.

\begin{lem}\label{lem:gen}
Let $n$ be an integer such that $n \geq 2$, and $p$ an odd prime. Let $V \subseteq \mathbb{F}_{p^{{2n}}}$ be an $\mathbb{F}_{p}$-space of dimension $n$, with $1 \in V$, and $V \neq \mathbb{F}_{p^{n}}$. Then there is an element $v\in V$ such that $v$ does not lie in any proper subfield of $\F_{p^{2n}}$.
\end{lem}
\begin{proof}
Suppose otherwise that $V \subseteq \bigcup_{\substack{d\mid 2n, d<2n}} \F_{p^d}$. Since $V \neq \mathbb{F}_{p^{n}}$, the subspace $V\cap \F_{p^{n}}$ has dimension at most $n-1$ over $\F_p$. It follows that at least $p^{n}-p^{n-1}$ many elements of $V$ belong to $\bigcup_{\substack{d\mid 2n, d<n}} \F_{p^d}$. However, we have
\[
\bigg|\bigcup_{\substack{d\mid 2n, d<n}} \F_{p^d} \bigg| \leq \sum_{i=1}^{n-1} p^{i} = \frac{p^n-1}{p-1} - 1 <\frac{p^n-1}{2}<p^{n}-p^{n-1},
\]
a contradiction.
\end{proof}

The following lemma expresses the indicator function of a union
of cyclotomic classes in terms of multiplicative characters. This provides
the link between the geometric arguments in Sections~\ref{sec:main1} and~\ref{sec:main2} and the character-sum estimates above.

\begin{lem}
\label{lem:indicator}
Let $K$ be a finite field and \(H\) be a multiplicative subgroup of \(K\) of index \(d\). Let
\(g\) be a primitive root of \(K\) and 
\[
    D=\bigcup_{k=1}^{r} g^{m_k}H,
\]
where \(m_1,\ldots,m_r\) are distinct elements of
\(\{0,1,\ldots,d-1\}\). Let \(\theta=\exp(2\pi i/d)\), and let
\(\chi\) be the multiplicative character of \(K\) of order
\(d\) satisfying \(\chi(g)=\theta\). Then the indicator function $\mathbf 1_D$ of the set $D$ is given by
\[
    \mathbf 1_D
    =
    \frac{1}{d}\sum_{j=0}^{d-1}
    \left(\sum_{k=1}^{r}\theta^{-jm_k}\right)\chi^j.
\]
\end{lem}

\begin{proof}
By orthogonality relations, for each \(1\leq k\leq r\) and
\(x\in K\), we have
\[
    \mathbf 1_{g^{m_k}H}(x)
    =
    \frac{1}{d}\sum_{j=0}^{d-1}
    \theta^{-jm_k}\chi^j(x).
\]
Summing over \(k\) gives the desired identity.
\end{proof}

The next lemma bounds the coefficients occurring in the expansion above.

\begin{lem}\label{lem:CS}
Let $d,r$ be positive integers with $r\leq d$. Let $m_1,m_2,\ldots, m_r$ be distinct integers in $\{0,1,\ldots, d-1\}$. Let $\theta=\exp(2\pi i/d)$. Then 
$$
\sum_{j=0}^{d-1} \bigg|\sum_{k=1}^{r} \theta^{-jm_k}\bigg|\leq d\sqrt{r}.
$$
\end{lem}
\begin{proof}
By orthogonality relations, we have
$$
 \sum_{j=0}^{d-1} \bigg|\sum_{k=1}^{r} \theta^{-jm_k}\bigg|^2=\sum_{j=0}^{d-1} \sum_{k=1}^{r} \sum_{\ell=1}^{r} \theta^{j(m_k-m_{\ell})}=\sum_{k=1}^{r} \sum_{\ell=1}^{r} \sum_{j=0}^{d-1}\theta^{j(m_k-m_{\ell})}=dr.
$$
Thus, the Cauchy–Schwarz inequality implies that
\[
\sum_{j=0}^{d-1} \bigg|\sum_{k=1}^{r} \theta^{-jm_k}\bigg|\leq \sqrt{dr \cdot d}=d\sqrt{r}. \qedhere
\]    
\end{proof}

\section{Proof of Theorem~\ref{thm:main}}\label{sec:main1}
Let $f:\F_q \to \F_q$ be a function such that $\mathcal{D}_f \subseteq D$ and $|\mathcal{D}_f|\leq \frac{q+1}{2}$.

Let $L=\{x^{p-1}:x \in \F_q^*\}$. Since $p\geq 3$ and $(p-1)$ is a divisor of $(q-1)$, $L$ is a proper subgroup of $\F_q^*$. By replacing $f$ with $f-f(0)$, we may assume that $f(0)=0$. By replacing $f$ with $f
/f(1)$ and replacing $D$ with $D/f(1)$, we may further assume that $f(1)=1$. Then $1\in \mathcal{D}_f$.

We first deal with a special case.
\begin{claim}
If
\begin{equation}\label{eq:HcupL}
\frac{\mathcal{D}_f}{\mathcal{D}_f} \subseteq H \cup L,
\end{equation}
then there exists $j\in \{0,1,\ldots, n-1\}$, such that $f(x)=x^{p^j}$ for all $x\in \F_q$.
\end{claim}
\begin{poc}
We may assume that $\mathcal{D}_f \not \subseteq L$ and $\mathcal{D}_f \not \subseteq H$. Indeed, if $\mathcal{D}_f\subseteq L$, then we can apply the Carlitz--McConnel theorem (Theorem~\ref{thm:CM}) or Theorem~\ref{thm:Y} to obtain the desired conclusion; by the same reason, if $\mathcal{D}_f\subseteq H$, then we are also done.

Since $1\in \mathcal{D}_f$, inclusion~\eqref{eq:HcupL} implies that  
\begin{equation}\label{eq:eq1}
\mathcal{D}_f \subseteq H \cup L.
\end{equation}
Since $\mathcal{D}_f \not \subseteq L$, 
we can find $a\in \mathcal{D}_f$ such that $a \in H$ and $a\notin L$. Then inclusion~\eqref{eq:HcupL} also implies that $\mathcal{D}_f/a \subseteq H \cup L$, that is, 
\begin{equation}\label{eq:eq2}
\mathcal{D}_f \subseteq a(H \cup L)=H \cup aL.
\end{equation}
However, since $L \cap aL=\emptyset$, it follows from inclusions~\eqref{eq:eq1} and~\eqref{eq:eq2} that \[\mathcal{D}_f \subseteq (H \cup L) \cap (H \cup aL)=H,\]
violating the assumption.
\end{poc}

In the rest of the proof, assume that inclusion~\eqref{eq:HcupL} fails to hold. Then we can pick $a,b\in \mathcal{D}_f$ such that $a/b\notin H \cup L$. Say we have $u_1,u_2,v_1,v_2\in \F_q$ such that $u_1\neq u_2$ and $v_1\neq v_2$, such that
\[
a=\frac{f(u_1)-f(u_2)}{u_1-u_2}, \qquad b=\frac{f(v_1)-f(v_2)}{v_1-v_2}.
\]
Since $|\mathcal{D}_f|\leq \frac{q+1}{2}$, by Theorem~\ref{thm:directions}, $f$ is additive. Thus, by setting $u=u_1-u_2$ and $v=v_1-v_2$, we have
\[
a=\frac{f(u)}{u}, \qquad b=\frac{f(v)}{v}.
\]
Note that $u/v\notin \F_p^*$, for otherwise $a=b$ by the additive property of $f$. Thus, it follows that for each $\lambda\in \F_p$, we have
\begin{equation}\label{eq:inD}
\frac{au-b\lambda v}{u-\lambda v}=\frac{f(u)-f(\lambda v)}{u-\lambda v} \in \mathcal{D}_f \subseteq D.
\end{equation}

Let $g$ be a primitive root of $\F_q$. Since $D$ is a union of $r$ cosets of $H$, we can find distinct integers $m_1, m_2, \ldots, m_r \in \{0,1,\ldots, d-1\}$, such that
\begin{equation}\label{eq:D}
D=\bigcup_{k=1}^{r} g^{m_k}H. 
\end{equation}
Let $\theta=\exp(2\pi i/d)$. Let $\chi$ be the multiplicative character of $\F_q$ with order $d$, such that $\chi(g)=\theta$. 

By equation~\eqref{eq:inD}, we have
$$
\sum_{\lambda \in\mathbb{F}_p} \mathbf 1_D\bigg(\frac{au-b\lambda v}{u-\lambda v}\bigg)=p.
$$
Thus, Lemma~\ref{lem:indicator} implies that
\begin{equation}\label{eq:p}
\frac{1}{d} \sum_{j=0}^{d-1} \bigg(\sum_{k=1}^{r} \theta^{-jm_k}\bigg) \sum_{\lambda \in\mathbb{F}_p} \chi^j\bigg(\frac{au-b\lambda v}{u-\lambda v}\bigg)=p.
\end{equation}

On the other hand, since $a/b\notin H \cup L$, the assumptions of Corollary~\ref{cor:charsum} are satisfied and thus for each $j\in \{1,2,\ldots, d-1\}$, it implies that 
\begin{equation}\label{eq:char}
\bigg |\sum_{\lambda \in\mathbb{F}_p}\chi^j \bigg(\frac{au-b\lambda v}{u-\lambda v}\bigg)\bigg|\le(2n-1)\sqrt p.    
\end{equation}
By comparing the estimates on character sums in \eqref{eq:p} and \eqref{eq:char}, it follows that
\begin{align}
p
&=\frac{rp}{d}+\frac{1}{d}\sum_{j=1}^{d-1} \bigg(\sum_{k=1}^{r} \theta^{-jm_k}\bigg) \sum_{\lambda \in\mathbb{F}_p} \chi^j\bigg(\frac{au-b\lambda v}{u-\lambda v}\bigg) \notag\\
&< \frac{rp}{d} +\frac{1}{d}\sum_{j=0}^{d-1} \bigg|\sum_{k=1}^{r} \theta^{-jm_k}\bigg| (2n-1)\sqrt{p}. \label{eq:bound}
\end{align}

Combining inequality~\eqref{eq:bound} and Lemma~\ref{lem:CS}, we obtain that
$$
\frac{(d-r)p}{d}<(2n-1)\sqrt{pr},
$$
that is, 
\[p<\frac{(2n-1)^2 rd^2}{(d-r)^2},\] violating the assumption.

\section{Proof of Theorem~\ref{thm:main2}}\label{sec:main2}

Since $q \equiv -1 \pmod d$ and $H$ is the subgroup of $\F_{q^2}^*$ with index $d$, it follows that $H$ is a union of cosets of $\F_q^*$. Since $S$ is a union of $r\leq d/2$ cosets of $H$, it follows that $S$ is a union of at most $\frac{q+1}{2}$ cosets of $\F_q^*$.

Let $A \subseteq \F_{q^2}$ such that $0,1\in A$, $|A|=q$, and $A-A \subseteq S \cup \{0\}$. Let $v \in \F_{q^2}^* \setminus S$. Since $\F_q^* \subseteq S$, the set $\{1,v\}$ forms a basis of $\F_{q^2}$ over $\F_q$. Consider the following embedding of $\F_{q^2}$ into $AG(2,q)$:
\[
\pi(a+bv)=(a,b), \qquad \forall a,b \in \F_q.
\]

Let $U=\pi(A)$. We claim that there is a function $f:\F_q \to \F_q$ such that
\[
U=\{(x,f(x)): x \in \F_q\}.
\]
Indeed, suppose otherwise that there are $x,y_1,y_2\in \F_q$ with $y_1\neq y_2$, such that $(x,y_1),(x,y_2)\in U$. Then we have $x+y_1v, x+y_2v\in A$ and thus $(y_1-y_2)v \in S \cup \{0\}$. However, since $S$ is a union of cosets of $\F_q^*$, this implies that $v \in S$, a contradiction.

Let $x_1, x_2\in \F_q$ with $x_1\neq x_2$. Then $x_1+f(x_1)v, x_2+f(x_2)v\in A$ and thus \[(x_1-x_2)+(f(x_1)-f(x_2))v\in S.\] Since $S$ is a union of cosets of $\F_q^*$, it follows that
\[
1+\frac{f(x_1)-f(x_2)}{x_1-x_2} \cdot v\in S.
\]
Thus,
\begin{equation}\label{eq:direction}
1+\mathcal{D}_f \cdot v \subseteq  S \cap (1+\F_q v).
\end{equation}
Observe that each coset of $\F_q^*$ intersects with $(1+\F_q v)$ in at most one element. Thus, equation~\eqref{eq:direction} implies that$|\mathcal{D}_f|\leq \frac{q+1}{2}$. Theorem~\ref{thm:directions} then implies that $f$ is additive. In particular, \[A=\{x+f(x)v:x\in \F_q\}\] is a subspace of $\F_{q^2}$ over $\F_p$. Since $1\in A$, we must have $f(1)=0$.

For the sake of contradiction, assume otherwise that $A\neq \F_q$. By Lemma~\ref{lem:gen}, there is $b\in A$ such that $b$ does not lie in any proper subfield of $\F_{q^2}$. By definition, we can write $b=a+f(a)v$ for some $a\in \F_q$. Since $f$ is additive and $f(1)=0$, it follows that for each $\lambda\in \F_p$,
\[
\frac{f(a)}{a-\lambda}=\frac{f(a)-f(\lambda)}{a-\lambda}\in \mathcal{D}_f.
\]
Thus, by equation~\eqref{eq:direction}, for each $\lambda\in \F_p$,
\[
1+\frac{f(a)}{a-\lambda}v=\frac{a-\lambda+f(a)v}{a-\lambda}=\frac{b-\lambda}{a-\lambda} \in S.
\]

Let $g$ be a primitive root of $\F_{q^2}$. Since $S$ is a union of $r$ cosets of $H$, we can find distinct integers $m_1, m_2, \ldots, m_r \in \{0,1,\ldots, d-1\}$, such that
\[
S=\bigcup_{k=1}^{r} g^{m_k}H. 
\]
Let $\theta=\exp(2\pi i/d)$. Let $\chi$ be the multiplicative character of $\F_{q^2}$ with order $d$, such that $\chi(g)=\theta$. Similar to the proof of Theorem~\ref{thm:main}, applying Lemma~\ref{lem:indicator} with $K=\F_{q^2}$, we have
\begin{equation}\label{eq:p2}
\frac{1}{d} \sum_{j=0}^{d-1} \bigg(\sum_{k=1}^{r} \theta^{-jm_k}\bigg) \sum_{\lambda \in\mathbb{F}_p} \chi^j\bigg(\frac{b-\lambda}{a-\lambda}\bigg)=p.
\end{equation}

On the other hand, since $a\in \F_{q}^*$ and $b\in \F_{q^2}$ such that $b$ does not lie in any proper subfield of $\F_{q^2}$, for each $j\in \{1,2,\ldots, d-1\}$, Corollary~\ref{cor:charsum2} implies that 
\begin{equation}\label{eq:char2}
\bigg |\sum_{\lambda \in\mathbb{F}_p}\chi^j \bigg(\frac{b-\lambda}{a-\lambda}\bigg)\bigg|\le(4n-1)\sqrt p.    
\end{equation}
By comparing the estimates on character sums in \eqref{eq:p2} and \eqref{eq:char2} and applying Lemma~\ref{lem:CS}, we have 
\begin{align*}
p
&=\frac{rp}{d}+\frac{1}{d}\sum_{j=1}^{d-1} \bigg(\sum_{k=1}^{r} \theta^{-jm_k}\bigg) \sum_{\lambda \in\mathbb{F}_p} \chi^j\bigg(\frac{b-\lambda}{a-\lambda}\bigg)\\
&<\frac{rp}{d}+\frac{1}{d}(4n-1)\sqrt{p} \cdot \sum_{j=0}^{d-1} \bigg|\sum_{k=1}^{r} \theta^{-jm_k}\bigg|\leq \frac{rp}{d}+(4n-1)\sqrt{pr}.
\end{align*}
It follows that 
\[p<\frac{(4n-1)^2 rd^2}{(d-r)^2},\] violating the assumption. We conclude that $A=\F_q$.

\section*{Acknowledgments}
The first author was supported by the Research Grants Council (RGC) of Hong Kong (No. 16307524). The second author thanks the Hong Kong University of Science and Technology for hospitality during his visit, where this project was initiated. 
The authors thank the anonymous referees for their valuable comments and suggestions.

\bibliographystyle{abbrv}
\bibliography{main}

\end{document}